\documentclass[a4paper,11pt]{article}

\usepackage{graphicx}
\usepackage{amssymb}
\usepackage{amsmath}
\usepackage{amsfonts}
\usepackage{amsthm}
\usepackage{epsfig}
\usepackage[all]{xy}
\usepackage{mathdots}
\usepackage[english]{babel}

\newcommand{\bR}{\mathbb{R}}

\newcommand{\cH}{\mathcal{H}}
\newcommand{\cB}{\mathcal{B}}

\newcommand{\cP}{\mathcal{P}}

\newcommand{\bN}{\mathbb{N}}

\newcommand{\bE}{\mathbb{E}}

\newcommand{\mfg}{\mathfrak{g}}

\DeclareMathOperator{\Ker}{Ker}

\DeclareMathOperator{\osp}{\mathfrak{osp}}

\DeclareMathOperator{\msl}{\mathfrak{sl}}
\DeclareMathOperator{\msp}{\mathfrak{sp}}
\DeclareMathOperator{\so}{\mathfrak{so}}
\DeclareMathOperator{\spp}{\mathfrak{sp}}

\theoremstyle{plain}
\newtheorem{thm}{Theorem}
\newtheorem*{thmA}{Theorem A}
\newtheorem*{thmB}{Theorem B}

\theoremstyle{definition}

\def\ubar{\underline}
\newcommand{\pa}{\partial}
\DeclareMathOperator{\XX}{\mathbf X}

\title{Branching laws for spherical harmonics on superspaces in exceptional cases}
\author{Roman L\'{a}vi\v{c}ka\footnote{email: {\tt lavicka@karlin.mff.cuni.cz}}
}

\date{\small{
Charles University, Faculty of Mathematics and Physics,\\
Sokolovsk\'a 83, 186 75 Praha, Czech Republic}}

\begin{document}

\maketitle

{\small
\abstract{It turns out that harmonic analysis on the superspace $\bR^{m|2n}$ is quite parallel to the classical theory on
the Euclidean space $\bR^{m}$ unless the superdimension $M:=m-2n$ is even and non-positive. The underlying symmetry is given by the orthosymplectic superalgebra $\osp(m|2n)$.
In this paper, when the symmetry is reduced to $\osp(m-1|2n)$ we describe explicitly the corresponding branching laws for spherical harmonics on $\bR^{m|2n}$ also in exceptional cases, i.e, when $M-1\in-2\bN_0$. In unexceptional cases, these branching laws are well-known and quite analogous as in the Euclidean framework.
}

}

\section{Introduction}

Lie superalgebras and their representations play an important role in mathematics and physics. For an account of their mathematical theory, see \cite{Kac, CW}. Major applications in physics are supersymmetry and supergravity, see \cite{Mar} and references there.

As a~specific application in physics, let us mention the quantum Kepler problem on the superspace $\bR^{m|2n}$ studied in \cite{Zhang}. The underlying symmetry is given by the orthosymplectic superalgebra $\osp(m|2n)$. By describing the structure of polynomials on $\bR^{m|2n}$ under a~natural action of $\osp(m|2n)$, the energy eigenvalues and the bound states spectrum for the corresponding Schr\"odinger equation are determined but only when $M:=m-2n>1$. Here $M$ is the so-called superdimension of $\bR^{m|2n}$.  
In particular, to solve the problem two key ingredients are used. The first fact is that the $\osp(m|2n)$-module of polynomials on $\bR^{m|2n}$ is completely reducible when $M>1$, and the second one is the branching law for spherical harmonics under the same condition $M>1$, see \cite[Appendix A]{Zhang}. 
In general, as we recall later on in detail, the first fact is not valid (see Theorems A and B below)
and, as we show in this paper, the general branching law is much more complicated, see Theorems 2 and 3 below.

In the framework of Clifford analysis, F. Sommen together with his collaborators has developed function theory on the superspace $\bR^{m|2n}$
which deals with polynomials or functions depending on $m$ commuting
(bosonic) and $2n$ anticommuting (fermionic) variables \cite{Cou,Cou_JPAA,CD,CDS,DES,DS3}, cf.\ \cite{LX}.  
The full symmetry is governed by the Howe dual pair $(\osp(m|2n),\msl(2))$, see \cite{Cou}. For an account of Howe dualities, we refer to \cite{Howe1,Howe2,CW,LeiS}. 
Early, it turned out that harmonic analysis on the superspace $\bR^{m|2n}$ is quite parallel to the classical theory on
the Euclidean space $\bR^{m}$ but only in the case when $M\not\in-2\bN_0$. 
In particular, when $M\not\in-2\bN_0$ 
the $\osp(m|2n)$-module of polynomials on $\bR^{m|2n}$ is completely reducible, see \cite{DS3, Cou}, generalizing the result obtained in \cite{Zhang} just for $M>1$  we mentioned above.  
Indeed, when $M\not\in-2\bN_0$, it is shown that the Fischer decomposition is an irreducible decomposition of homogeneous polynomials on $\bR^{m|2n}$ under the
action of $\osp(m|2n)$, see \cite{Cou}, cf.\ \cite{Zhang}.

It is well-known that, in the classical case (i.e., when $n=0$), a~contruction of the so-called Gelfand-Tsetlin (GT) bases of the spaces $\cH_k(\bR^m)$ of $k$-homogeneous spherical harmonics on $\bR^m$ is based on knowledge of the Fischer decompositions. Indeed, this enables to describe explicitly branching laws of the spaces $\cH_k(\bR^m)$ when the symmetry given by Lie algebra $\so(m)$ is reduced to $\so(m-1)$, and then, 
for a~given chain of Lie algebras
$$\so(m)\supset\so(m-1)\supset\cdots\supset\so(2),$$
the construction of the corresponding GT bases,  see e.g.\ \cite{Lav} for more details. 
On the other hand, in the super case, when the symmetry given by $\osp(m|2n)$ is reduced to $\osp(m-1|2n)$ branching laws for spherical harmonics on $\bR^{m|2n}$ are given in \cite{Cou}, cf.\ \cite{Zhang}, just in unexceptional cases, i.e, when $M-1\not\in-2\bN_0$.   

In \cite{LavS},  for the first time, the Fischer decomposition for scalar valued polynomials on $\bR^{m|2n}$ was obtained also in the exceptional cases, i.e, when $M\in-2\bN_0$ and $m\not=0$. In the exceptional cases, the Fischer
decomposition is rather an indecomposable (but not necessarily irreducible) decomposition of homogeneous polynomials under the
action of $\osp(m|2n)$, see Theorems A and B below. 
In the general case, the $\osp(m|2n)$-module of polynomials on $\bR^{m|2n}$ was also described recently in \cite{She}.

In this paper, when the symmetry is reduced to $\osp(m-1|2n)$ we describe branching laws for
$\osp(m|2n)$-modules of homogeneous (generalized) spherical harmonics 
on $\bR^{m|2n}$ also in the exceptional cases, i.e, when $M-1\in-2\bN_0$, see Theorems 2 and 3 below.  
To do this we apply the Fischer decompositions in the exceptional cases and the Cauchy-Kovalevskaya extension for polynomials on superspaces, see Theorem 1.
These general branching laws enable to construct GT bases of the spaces $\cH_k(\bR^{m|2n})$ of $k$-homogeneous spherical harmonics on $\bR^{m|2n}$
for a~given chain of Lie superalgebras
$$\osp(m|2n)\supset\osp(m-1|2n)\supset\cdots\supset\osp(0|2n)=\spp(2n),$$
similarly, as in the classical case, see Section \ref{s_GTB}. We plan to investigate properties of the GT bases of $\cH_k(\bR^{m|2n})$ in the next paper. 

In \cite{CD}, super Dirac operator is
introduced and the corresponding Fischer decomposition is established for spinor valued polynomials on
$\bR^{m|2n}$ but again only when $M\not\in-2\bN_0$.
We believe that it is possible to generalize the results obtained in this paper to the spinor valued polynomials on $\bR^{m|2n}$.

\section{Harmonic analysis on superspaces}


In this section, we briefly recall  some known facts from harmonic analysis on superspaces, see \cite{Cou} and \cite{LavS} for more details. 

Let $V=V_0\oplus V_1$ be a~finite dimensional superspace endowed with a~scalar superproduct $g$, 
that is, $g$ is a non-degenerate bilinear form on $V$ such that $g|_{V_0\times V_0}$ is symmetric, $g|_{V_1\times V_1}$ is skew symmetric
and $g|_{V_0\times V_1}=0=g|_{V_1\times V_0}$. We view $V$ as the defining representation for the Lie superalgebra $\mfg=\osp(V,g)$ and we identify the supersymmetric tensor algebra of $V$ with the space $\cP(V)$ of scalar-valued polynomials on $V$. There is a~natural action of $\mfg$ on the superpolynomials $\cP(V)$ and it turns out that the hidden symmetry is given by the Lie algebra $\msl(2)$ of invariant operators on $\cP(V)$ which is generated by the superlaplacian $\Delta$ and the square norm $R^2$ of the supervariable $X\in V$ (that is, $R^2=g(X,X)$).

To be more explicit, let $V=\bR^{m|2n}$ and we denote the standard coordinates of the supervariable $X\in\bR^{m|2n}$ as
$$X=(X_1,\ldots,X_{m+2n})=(x_1,\ldots,x_m,\theta_1,\ldots,\theta_{2n}).$$
Then the space $\cP=\cP(\bR^{m|2n})$ of polynomials on $\bR^{m|2n}$ is given as
$$\cP=\bR[x_1,\ldots,x_m]\otimes\Lambda(\theta_1,\ldots,\theta_{2n})$$
where $\bR[x_1,\ldots,x_m]$ are the polynomials in $m$ commuting (bosonic) variables $x_1,\ldots,x_m$ and $\Lambda(\theta_1,\ldots,\theta_{2n})$ is the Grassmann algebra generated by $2n$ anticommuting (fermionic) variables $\theta_1,\ldots,\theta_{2n}$.
We assume that the matrix of the scalar superproduct $g$ is the block diagonal matrix
$$g=(g^{ij})=
\begin{pmatrix}
E_m&\ \\
\ &J_{2n}\\
\end{pmatrix}
$$
where $E_m$ is the identity matrix of size $m$ and $J_{2n}$ is the square matrix of size $2n$ given by
$$
J_{2n}=\frac{1}{2}
\begin{pmatrix}
0&-1&&&& \\
1&0&&&\\
&&&\ddots&&\\
&&&&0&-1\\
&&&&1&0
\end{pmatrix}
$$
As usual we write $\osp(m|2n)$ for $\osp(V,g)$.

The basic $\osp(m|2n)$-invariant operators on $\cP$ are 
$$\Delta=\sum_{j=1}^{m} \pa_{x_j}^2-4\sum_{j=1}^{n}\pa_{\theta_{2j-1}}\pa_{\theta_{2j}},\ 
R^2=\sum_{j=1}^{m} {x^2_j}-\sum_{j=1}^{n}{\theta_{2j-1}}{\theta_{2j}},$$ 
$$\bE=\sum_{j=1}^{m} x_j\pa_{x_j}+\sum_{j=1}^{2n}\theta_j\pa_{\theta_j}.$$ 
Here we have the supercommutation relations $[\pa_{X_i},X_j]=\delta_{ij}$ for $1\leq i,j\leq m+2n$,
 $\Delta$ is the super Laplace operator and $\bE$ is the super Euler operator on $\bR^{m|2n}$. These operators generate the Lie algebra $\msl(2)$, that is, the relations 
$[\Delta/2, R^2/2]=\bE+M/2$, $[\Delta/2, \bE+M/2]=2\Delta/2$ and  $$[R^2/2, \bE+M/2]=-2R^2/2$$
hold true
where $M=m-2n$ is the so-called superdimension of $\bR^{m|2n}$. See \cite{DES,DS3,Cou} for details.

Using the operator $\bE$ we define the space of $k$-homogeneous polynomials by
$$\cP_k=\{P\in\cP|\ \bE P=kP\}.$$ For an operator $A$ on $\cP$, we denote
$$\Ker A=\{H\in\cP|\ A H=0\}\text{\ \ and\ \ }\Ker_k A=\Ker(A)\cap\cP_k.$$ 
Then the space $\cH_k=\cH_k(\bR^{m|2n})$ of $k$-homogeneous spherical harmonics is given by
$$\cH_k=\Ker_k \Delta.$$ 

\paragraph{The Fischer decompositions.} 
(I) In \cite[Theorem 3]{DS3}, the Fischer decomposition of polynomials on $\bR^{m|2n}$ into spherical harmonics is obtained unless  
$M\in-2\bN_0$. In this case, for each $k\in\bN_0$, we have that $\cP_k = \cH_k \oplus R^2 \cP_{k-2}$  and thus
\begin{equation}\label{e_fischer}
\cP_{k}=\bigoplus_{j=0}^{\lfloor k/2\rfloor} R^{2j}\cH_{k-2j}.
\end{equation}
Hence the Fischer decomposition looks like the classical (purely bosonic) one. In Figure \ref{Fig_I}, all the summands of $\cP_k$ are contained in the $k$-th column. Each row yields an infinite dimensional representation of $\msl(2)$.

Let us remark that, 
in the case  $m=1$, we have  $\cH_k\not=0$  if and only if $k=0,\ldots,2n+1$ (see \cite{Cou}). So in this case,  for all $k>2n+1$, the $k$-th rows in Figure \ref{Fig_I} are missing.

\begin{figure}[h]

\caption{The Fischer decomposition for $M\not\in -2\bN_0$.}\label{Fig_I}

$$\xymatrix@!C=5pt@R=1pt{
\cP_0 & \cP_1 & \cP_2 & \cP_3 & \cP_4 & \cP_5 & \cdots \\
\\
\cH_0 & &  R^2\cH_0 & & R^4\cH_0 & & \cdots \\
& \cH_1 & &  R^2\cH_1 & & R^4\cH_1  \\
& & \cH_2 & & R^2\cH_2 & & \cdots   \\
& & & \cH_3 & &  R^2\cH_3  \\
& & & & \cH_4 & & \cdots \\
& & & & & \cH_5 & & \\
& & & & & & \ddots & & \\
}$$

\end{figure}

\smallskip\noindent
(II) The purely fermionic case when $m=0$ is well-known, see e.g.\ \cite[\S\;4.2]{DES}. Actually, in this case, $\cP_k=\Lambda^k(\theta_1,\ldots,\theta_{2n})$ is the $k$-th homogeneous part of the Grassmann algebra $\Lambda(\theta_1,\ldots,\theta_{2n})$ and the Fischer decomposition is a well-known irreducible decomposition of $\Lambda^k(\theta_1,\ldots,\theta_{2n})$ under the action of $sp(2n)$. Indeed, $\cH_k=\cH_k(\bR^{0|2n})$, 
$$R^2=-\sum_{j=1}^{n}{\theta_{2j-1}}{\theta_{2j}}$$ and, for $k=0,\ldots, n,$ we have
\begin{equation}\label{e_fisher_sympl}
\begin{split}
\cP_{k}&=\bigoplus_{j=0}^{\lfloor k/2\rfloor} R^{2j}\cH_{k-2j},\\
\cP_{2n-k}&=\bigoplus_{j=0}^{\lfloor k/2\rfloor} R^{2n-2k+2j}\cH_{k-2j}.
\end{split}
\end{equation}
Hence, for $M=-2n$, the Fischer decomposition is depicted in Figure \ref{Fig_II}. In this case, the diagram has $n+1$ rows and each row forms a~finite dimensional $\msl(2)$-representation. 
In Section \ref{s_GTB} (I), we recall a~construction of bases of $\cH_k(\bR^{0|2n})$ given in \cite[\S\;4.2]{DES}.

\begin{figure}[h]

\caption{The Fischer decomposition for $M\in -2\bN_0$ and $m=0$.}\label{Fig_II}

$$\xymatrix@!C=5pt@R=1pt{
\cP_0 & \cP_1 & \cP_2 & \cP_3 & \cdots & \cdots & \cP_{2n-2} \ \ \ & \cP_{2n-1} & \cP_{2n} \\
\\
\cH_0 & &   R^2\cH_0 & & \cdots & & R^{2(n-1)}\cH_0 & & R^{2n}\cH_0 \\
&    \cH_1 & & R^2\cH_1 & & \cdots & & R^{2(n-1)}\cH_1 \\
& & \cH_2 & & \cdots & & R^{2(n-2)}\cH_2 \\
& & & \ddots & & \iddots \\
& & & & {\cH}_n\\
}$$

\end{figure}

\smallskip\noindent
(III) In \cite{LavS}, the Fischer decomposition for polynomials on $\bR^{m|2n}$ is described even in the exceptional case 
when $M\in-2\bN_0$ and $m\not=0$.
In this case, denote the set of exceptional indices by
\begin{equation}\label{IM}
I_M=\{k\in\bN_0|\ 2-M/2\leq k\leq 2-M\}.
\end{equation}
Then, for $k\in I_M$, it is known that
$\cH_k \cap R^2 \cP_{k-2}\not=\emptyset$ and so the decomposition \eqref{e_fischer} cannot be valid.
In the general case, the following decomposition holds true, see \cite{LavS}, 
$\cP_k = \tilde{\cH}_k \oplus R^2 \Delta R^2 \cP_{k-2}$
where $\tilde{\cH}_k = \Ker_k (\Delta R^2 \Delta)$. Thus we have that
\begin{equation}\label{e_superFischer}
\cP_{k}=\bigoplus_{j=0}^{\lfloor k/2\rfloor} (R^2 \Delta R^2)^j\tilde\cH_{k-2j}.
\end{equation}

Now we recall structure of $\osp(m|2n)$-modules $\tilde\cH_k$ and $\cH_k$.
Denote by $L^{m|2n}_{\lambda}$ an $\osp(m|2n)$-irreducible module with the highest weight $\lambda$. We use the simple root system of $\osp(m|2n)$ as in \cite{Cou,Zhang}. This is not the standard choice \cite{Kac} but is more convenient for our purposes.

\begin{thmA}{\rm (\cite{Cou,LavS})} Let $I_M=\{k\in\bN_0|\ 2-M/2\leq k\leq 2-M\}$ if $M\in-2\bN_0$ and $I_M=\emptyset$ otherwise. For $k\in\bN_0$, denote $\cH^0_k=\cH_k \cap R^2 \cP_{k-2}$ and $\tilde{\cH}_k = \Ker_k (\Delta R^2 \Delta)$.

\smallskip\noindent
(i) If $k\not\in I_M$ then $\tilde{\cH}_k=\cH_k\simeq L^{m|2n}_{(k,0\ldots,0)}$  and $\cH^0_k=0$. 

\smallskip\noindent
(ii) Let $k\in I_M$. Then $\cH^0_k=R^{2k+M-2}\cH_{2-M-k}$ and the indecomposable  $\osp(m|2n)$-module
$\tilde\cH_k$ has a~composition series $\cH^0_k\subset\cH_k\subset\tilde\cH_k$ with the irreducible quotients
$$\cH^0_k\simeq L^{m|2n}_{(2-M-k,0\ldots,0)},\ \ \cH_k/\cH^0_k\simeq L^{m|2n}_{(k,0\ldots,0)}, \ \ \tilde\cH_k/\cH_k\simeq L^{m|2n}_{(2-M-k,0\ldots,0)}.$$ 
Here e.g.\ $\cH_k/\cH^0_k$ is the quotient of vector spaces endowed with a~natural action of $\osp(m|2n)$.
\end{thmA}

Let us note that 
some direct summands of the decomposition \eqref{e_superFischer} might be trivial. Indeed, we have

\begin{thmB}{\rm (\cite{LavS})}\label{cor_Fischer}  Let $k\in\bN_0$,
$N_k=\{k-2j|\ j=0,\ldots,\lfloor k/2\rfloor\}$ and $\tilde J_k=N_k\cap I_M$. 
Then, under the action of $\osp(m|2n)$, $\cP_k$ has an indecomposable decomposition
\begin{equation}\label{eq_Fischer+}
\cP_{k}=\bigoplus_{\ell\in \tilde J_k} R^{k-\ell}\tilde\cH_{\ell} \oplus
\bigoplus_{\ell\in J_k} R^{k-\ell}\cH_{\ell} 
\end{equation}
where $J_k=N_k\setminus(\tilde J_k\cup J_k^0)$ with $J_k^0=\{2-M-\ell|\ \ell\in\tilde J_k\}$. 
\end{thmB}

The particular case $M=-4$ is depicted in Figure \ref{Fig_III}. The exceptional indices are $I_{-4}=\{4,5,6\}$.
Notice that the first three rows look like the diagram for the purely fermionic Fischer decomposition  with $n=2$ (see Figure \ref{Fig_II}).
Other rows are infinite as in the classical case (see Figure \ref{Fig_I}). The $k$-th row starts with $\cH_k$ except for the exceptional indices $k\in I_{-4}$ when it starts with ${\tilde{\cH}_k}$.

\begin{figure}[h]

\caption{The Fischer decomposition for $M=-4$ and $m\not=0$.}
\label{Fig_III}

$$\xymatrix@!C=5pt@R=1pt{
\cP_0 & \cP_1 & \cP_2 & \cP_3 & \cP_4 & \cP_5 & \cP_6 & \cP_7 & \cP_8 & \cdots\\
\\
\cH_0 & &  R^2\cH_0 & & R^4\cH_0 & & 0 & &  \\
& \cH_1 & & R^2\cH_1 & & 0 & &  & \\
& & \cH_2 & & 0 & &  & &  \\
& & & \cH_3 & & R^2\cH_3 & & R^4\cH_3 & & \cdots\\
& & & & {\tilde{\cH}_4} & &  {R^2\tilde{\cH}_4} & & {R^4\tilde{\cH}_4} \\
& & & & & {\tilde{\cH}_5} & & {R^2\tilde{\cH}_5} & & \cdots\\
& & & & & & {\tilde{\cH}_6} & & {R^2\tilde{\cH}_6}\\
& & & & & & & \cH_7 & & \cdots\\
& & & & & & & & \cH_8 \\
& & & & & & & & & \ddots
}$$
\end{figure}

\section{The Cauchy-Kovalevskaya Extension}\label{s_CK}

For an explicit description of branching laws for (generalized) spherical harmonics on $\bR^{m|2n}$ when the symmetry is restricted from $\osp(m|2n)$ to $\osp(m-1|2n)$ we need to generalize an algebraic version of the so-called Cauchy-Kovalevskaya extension to the super setting. In particular, we show that each polynomial $Q_k$ on $\bR^{m|2n}$ is uniquely determined by its laplacian $\Delta Q_k$ and the values of $Q_k$ and $\pa_{x_m} Q_k$ on the hyperplane $x_m=0$. We identify the hyperplane $x_m=0$ with $\bR^{m-1|2n}$ and write $\ubar\cP=\cP(\bR^{m-1|2n})$ for the polynomials in the supervariable  $$\ubar X=(x_1,x_2,\ldots,x_{m-1},\theta_1,\theta_2,\ldots, \theta_{2n})\in\bR^{m-1|2n}.$$
Indeed, we have the following result.

\begin{thm}\label{t_CK} 
\begin{itemize}
\item[(i)] For each $p_{k}\in\ubar\cP_{k}$, $p_{k-1}\in\ubar\cP_{k-1}$, $P_{k-2}\in\cP_{k-2}$, there is a~unique $Q_k\in\cP_k$ such that
\begin{equation}\label{IVP}
\Delta Q_k=P_{k-2},\ \ (Q_k|_{x_m=0})=p_k,\ \ (\pa_{x_m} Q_k|_{x_m=0})=p_{k-1}.
\end{equation}
Put $CK(p_k,p_{k-1},P_{k-2}):=Q_k$.
\item[(ii)] Then the Cauchy-Kovalevskaya extension operator $CK$ is  an invariant isomorphism of $\ubar\cP_{k}\oplus\ubar\cP_{k-1}\oplus\cP_{k-2}$ onto $\cP_k$ under the action of $\osp(m-1|2n)$.
In addition, we have
\begin{equation}\label{eq-CK}
CK(p_k,p_{k-1},P_{k-2})=\sum^k_{\ell=0} \XX_{\ell} p_{k-\ell}
\end{equation}
where $p_{k-2-\ell}=(\pa^{\ell}_{x_m} P_{k-2})|_{x_m=0}$ for $\ell=0,\ldots,k-2$ and
\begin{equation}\label{eq-Xl}
\XX_{\ell}=\sum_{j=0}^{\infty} \frac{x_m^{2j+\ell}}{(2j+\ell)!}(-\ubar\Delta)^j.
\end{equation}
Here $\ubar\Delta$ is the Laplace operator on $\bR^{m-1|2n}$.
\end{itemize}
\end{thm}

\begin{proof}
(a) It is clear that we can write polynomials $Q_k\in\cP_k$ and $P_{k-2}\in\cP_{k-2}$ in the form
\begin{equation}\label{eq-Qk}
Q_k(X)=\sum_{j=0}^k \frac{x_m^j}{j!} q_{k-j}(\ubar X)\text{\ \ and\ \ }
P_{k-2}(X)=\sum_{j=0}^{k-2} \frac{x_m^j}{j!} p_{k-2-j}(\ubar X)
\end{equation}
for some polynomials $p_j,q_j\in\ubar \cP_j$.
Then $Q_k$ is a~unique solution of the initial value problem \eqref{IVP} if and only if 
$q_k=p_k$, $q_{k-1}=p_{k-1}$ and
\begin{equation}\label{IVP_rel}
 q_{k-2-j}=-\ubar\Delta q_{k-j}+p_{k-2-j} \text{\ for\ }j=0,\ldots,k-2.
\end{equation}
Indeed, we have that 
\begin{equation}\label{eq-DQk}
\Delta Q_k=(\pa^2_{x_m}+\ubar\Delta)Q_k=\sum_{j=2}^k \frac{x_m^{j-2}}{(j-2)!} q_{k-j}+\sum_{j=0}^{k-2} \frac{x_m^j}{j!} \ubar\Delta q_{k-j}.
\end{equation}
Then the equation $\Delta Q_k=P_{k-2}$ is equivalent to the recursion relation \eqref{IVP_rel} by comparing terms with the same powers $x_m^j$ 
in the expansions of $P_{k-2}$ and $\Delta Q_k$, see \eqref{eq-Qk} and \eqref{eq-DQk}.

\noindent 
(b) For a~given $\ell=0,\ldots,k$, assume that all the initial data $p_{k-j}$ vanish except possibly for $p_{k-\ell}$. 
Then, by  \eqref{IVP_rel}, it is easy to see that the solution $Q_k$ of the problem \eqref{IVP} is given by $Q_k=\XX_{\ell} p_{k-\ell}$.
Here the operator $\XX_{\ell}$ is defined as in \eqref{eq-Xl}. To finish the proof of the general case of the statement (ii), we use linearity of the problem \eqref{IVP}.
\end{proof}

\section{Branching laws for spherical harmonics}\label{s_BL}

 K. Coulembier \cite{Cou} (cf.\ \cite{Zhang}) showed that, under the action of $\osp(m-1|2n)$, the space  $\cH_k=\cH_k(\bR^{m|2n})$ of spherical harmonics decomposes as
\begin{equation}\label{eq_branch} 
\cH_{k}\simeq\bigoplus_{\ell=0}^k \ubar\cH_{\ell}
\end{equation}
provided the superdimension $\ubar M:=m-1-2n$ of $\bR^{m-1|2n}$ satisfies $\ubar M\not\in-2\bN_0$.
Here $\ubar\cH_{\ell}=\cH_{\ell}(\bR^{m-1|2n})$.
We extend these branching laws also to the exceptional cases.

\begin{thm}
Under the action of $\osp(m-1|2n)$, we have an indecomposable decomposition
$$\cH_{k}\simeq\bigoplus_{\ell\in\tilde B_k} \ubar{\tilde\cH}_{\ell} \oplus
\bigoplus_{\ell\in B_k} \ubar\cH_{\ell} $$
where $\tilde B_k=\{0,1,\ldots,k\}\cap I_{M-1}$ and $B_k=\{0,1,\ldots,k\}\setminus (\tilde B_k\cup B^0_k)$\\
with $B^0_k=\{3-M-\ell|\ \ell\in\tilde B_k\}$. Here $\ubar{\tilde\cH}_{\ell}=\tilde\cH_{\ell}(\bR^{m-1|2n})$.
\end{thm}

\begin{proof} 
By Theorem \ref{t_CK}, the Cauchy-Kovalevskaya extension operator $CK$ is an invariant isomorphism of $\ubar\cP_{k}\oplus\ubar\cP_{k-1}\oplus\cP_{k-2}$ onto $\cP_k$ under the action of $\osp(m-1|2n)$. It is obvious that
$$\cH_k=CK(\ubar\cP_{k}\oplus\ubar\cP_{k-1}\oplus 0).$$
Therefore, as $\osp(m-1|2n)$-modules, $\cH_k$ is isomorphic to $\ubar\cP_k\oplus\ubar\cP_{k-1}$. 
Now it is sufficient to decompose $\ubar\cP_k$ and $\ubar\cP_{k-1}$ using the Fischer decomposition \eqref{eq_Fischer+}.
Since $$\ubar R^2=\sum_{j=1}^{m-1} {x^2_j}-\sum_{j=1}^{n}{\theta_{2j-1}}{\theta_{2j}}$$ is an injective $\osp(m-1|2n)$-invariant operator we have $\ubar R^{2j}\ubar{\tilde\cH}_{\ell}\simeq\ubar{\tilde\cH}_{\ell}$ and $\ubar R^{2j}\ubar{\cH}_{\ell}\simeq\ubar{\cH}_{\ell}$.
\end{proof}

\begin{thm}
Let $M\in-2\bN_0$ and $k\in I_M=[2-M/2, 2-M]\cap \bN_0$.\\
Under the action of $\osp(m-1|2n)$, we have an irreducible decomposition
$$\tilde\cH_{k}\simeq\bigoplus_{\ell=0}^{2-M-k} 2\,\ubar{\cH}_{\ell}\ \oplus
\bigoplus_{\ell=3-M-k}^k \ubar\cH_{\ell} $$
\end{thm}

\begin{proof} 
Since $\tilde{\cH}_k = \Ker_k (\Delta R^2 \Delta)$ it is easily seen that $$\tilde\cH_k=CK(\ubar\cP_k\oplus\ubar\cP_{k-1}\oplus\Ker_{k-2}(\Delta R^2)).$$ Moreover, by Theorem A,  we know that $$\Ker_{k-2}(\Delta R^2)=R^{2k+M-4}\cH_{2-M-k}\simeq\cH_{2-M-k}.$$
Therefore, as $\osp(m-1|2n)$-modules, $\tilde\cH_k\simeq \ubar\cP_k\oplus\ubar\cP_{k-1}\oplus\cH_{2-M-k}$. 
Now it is sufficient to use the 'non-exceptional' Fischer decomposition \eqref{e_fischer} for $\ubar\cP_k$, $\ubar\cP_{k-1}$ and the branching law \eqref{eq_branch} for $\cH_{2-M-k}$
because the superdimension $\ubar M$ of $\bR^{m-1|2n}$ satisfies $\ubar M=M-1\not\in-2\bN_0$.
\end{proof}

\section{GT bases for spherical harmonics}\label{s_GTB}

Using the branching laws described in Section \ref{s_BL} it is now easy to construct GT bases of (generalized) spherical harmonics  on $\bR^{m|2n}$ for the chain of Lie superalgebras
$$\osp(m|2n)\supset\osp(m-1|2n)\supset\cdots\supset\osp(0|2n)=\msp(2n).$$
We shall proceed by induction on $m$.

\medskip\noindent
(I) In the purely fermionic case when $m=0$, a~construction of bases for spherical harmonics $\cH_k^{f}(n)=\cH_k(\bR^{0|2n})$, $k=0,\ldots,n$ are well known. By \cite[\S\;4.2]{DES}, we have the following decomposition
\begin{equation}\label{e_branch_sympl} 
\cH_k^{f}(n)=\cH_k^{f}(n-1)\oplus\theta_{2n-1}\cH_{k-1}^{f}(n-1)\oplus\theta_{2n}\cH_{k-1}^{f}(n-1)
 \oplus\Theta_{n,k}\cH_{k-2}^{f}(n-1)
\end{equation}
with $\Theta_{n,k}={\theta_{1}}{\theta_{2}}+\cdots+{\theta_{2n-3}}{\theta_{2n-2}}+(k-n-1){\theta_{2n-1}}{\theta_{2n}}$.
Here $\cH_k^{f}(n-1)$ is the space of fermionic spherical harmonics of degree $k$ in the Grassmann variables ${\theta_{1}},\cdots,{\theta_{2n-2}}$
and $\cH_k^{f}(n-1)=0$ unless $k=0,\ldots,n-1$. 
In addition, we have that $\cH_k^{f}(1)=\Lambda^k(\theta_1,\theta_{2})$ for $k=0,1$. Hence, using \eqref{e_branch_sympl}, it is now easy to construct a basis of $\cH_k^{f}(n)$ in an iterative way.

\medskip\noindent
(II) Suppose that $m\geq 1$.
In what follows, we use shorthand notations that $\cP_k=\cP_k(\bR^{m|2n})$, $\cH_k=\cH_k(\bR^{m|2n})$, $\ubar\cP_k=\cP_k(\bR^{m-1|2n})$, $\ubar\cH_k=\cH_k(\bR^{m-1|2n})$ and
$$\ubar R^2=\sum_{j=1}^{m-1} {x^2_j}-\sum_{j=1}^{n}{\theta_{2j-1}}{\theta_{2j}}.$$

\medskip\noindent
(a) Assume first $(M-1)\not\in -2\bN_0$. Then, for a~given  $k\in\bN_0$, we have the Fischer decomposition 
$$\ubar\cP_{k}=\bigoplus_{\ell\in N_k} \ubar R^{k-\ell}\ubar\cH_{\ell},$$
see \eqref{e_fischer}.
Suppose that we have constructed  bases $\ubar\cB_{\ell}=\{h_{\ell,\nu}|\ \nu\in \ubar T_{\ell}\}$ of the spaces $\ubar\cH_{\ell}$ for all $\ell\in\bN_0$. 

\smallskip\noindent
For $\ell\in N_k$ and $\nu\in \ubar T_{\ell}$, put
\begin{equation}\label{eCKharm1}
h_{k,\ell,\nu}=CK(\ubar R^{k-\ell} h_{\ell,\nu},0,0).
\end{equation}
In particular, we have $h_{\ell,\ell,\nu}=h_{\ell,\nu}$.

\smallskip\noindent
For $\ell\in N_{k-1}$ and $\nu\in \ubar T_{\ell}$, put
\begin{equation}\label{eCKharm2}
h_{k,\ell,\nu}=CK(0,\ubar R^{k-1-\ell} h_{\ell,\nu},0).
\end{equation}
Then, by Theorem \ref{t_CK}, it is easy to see that the polynomials \eqref{eCKharm1} and \eqref{eCKharm2} form
a~basis of the space $\cH_{k}$. 

\smallskip\noindent
Furthermore, if $M\in -2\bN_0$ and $k\in I_M$, then put, for $\ell=0,\ldots,2-M-k$ and $\nu\in \ubar T_{\ell}$,
\begin{equation}\label{eCKharm2+}
\tilde h_{k,\ell,\nu}=CK(0,0,R^{2k+M-4} h_{2-M-k,\ell,\nu}).
\end{equation}
Obviously, by Theorem \ref{t_CK}, the polynomials \eqref{eCKharm1}, \eqref{eCKharm2} and \eqref{eCKharm2+} form a~basis
of the space $\tilde\cH_{k}$.

\medskip\noindent
(b) Assume that $(M-1)\in -2\bN_0$. Then, for a~given  $k\in\bN_0$, we have
$$\ubar\cP_{k}=\bigoplus_{\ell\in \ubar{\tilde J}_k} \ubar R^{k-\ell}\ubar{\tilde\cH}_{\ell} \oplus
\bigoplus_{\ell\in\ubar J_k} \ubar R^{k-\ell}\ubar\cH_{\ell}$$ 
where $\ubar{\tilde J}_k=N_k\cap I_{M-1}$, $\ubar J_k=N_k\setminus(\ubar{\tilde J_k}\cup \ubar J_k^0)$ with $\ubar J_k^0=\{3-M-\ell|\ \ell\in\ubar{\tilde J_k}\}$, see Theorem B. 
Assume that we have constructed  bases $\ubar\cB_{\ell}=\{h_{\ell,\nu}|\ \nu\in \ubar T_{\ell}\}$ of the spaces $\ubar\cH_{\ell}$ for all $\ell\in\bN_0\setminus I_{M-1}$
and, for all $\ell\in I_{M-1}$,  bases $$\ubar{\tilde\cB}_{\ell}=\{\tilde h_{\ell,\nu}|\ \nu\in \tilde{\ubar T}_{\ell}\}$$ of the spaces $\ubar{\tilde\cH_{\ell}}$. 

\smallskip\noindent
For $\ell\in J_k$ and $\nu\in \ubar T_{\ell}$, put
\begin{equation}\label{eCKharm3}
h_{k,\ell,\nu}=CK(\ubar R^{k-\ell} h_{\ell,\nu},0,0).
\end{equation}

\smallskip\noindent
For $\ell\in J_{k-1}$ and $\nu\in \ubar T_{\ell}$, put
\begin{equation}\label{eCKharm4}
h_{k,\ell,\nu}=CK(0,\ubar R^{k-1-\ell} h_{\ell,\nu},0).
\end{equation}

\smallskip\noindent
For $\ell\in\tilde J_k$ and $\nu\in \tilde{\ubar T}_{\ell}$, put
\begin{equation}\label{eCKharm5}
\tilde h_{k,\ell,\nu}=CK(\ubar R^{k-\ell} \tilde h_{\ell,\nu},0,0).
\end{equation}

\smallskip\noindent
For $\ell\in \tilde J_{k-1}$ and $\nu\in \tilde{\ubar T}_{\ell}$, put
\begin{equation}\label{eCKharm6}
\tilde h_{k,\ell,\nu}=CK(0,\ubar R^{k-1-\ell} \tilde h_{\ell,\nu},0).
\end{equation}
Obviously, by Theorem \ref{t_CK}, the polynomials \eqref{eCKharm3}, \eqref{eCKharm4}, \eqref{eCKharm5} and \eqref{eCKharm6} form a~basis of the space $\cH_k$.

We plan to express elements of the GT bases in terms of classical special polynomials and investigate properties of the GT bases  in the next paper.

\subsection*{Acknowledgments}

The author is very grateful for useful advise and suggestions from Vladim\'ir Sou\v cek.
The support of the grant GACR 20-11473S is gratefully acknowledged.

\end{document}